\documentclass[11pt,fleqn]{scrartcl}
\usepackage[T1]{fontenc}
\usepackage[utf8]{inputenc}
\usepackage[english]{babel}
\usepackage{amsmath,amsfonts, amsthm,xcolor,amssymb}
\usepackage{paralist}
\usepackage{authblk}

\numberwithin{equation}{section}
\usepackage{mathtools, mathabx}
\usepackage[colorlinks=true,urlcolor=blue, citecolor=blue,linkcolor=blue,linktocpage,pdfpagelabels, bookmarksnumbered,bookmarksopen]{hyperref}
\usepackage{graphicx}
\usepackage{comment}
\usepackage[hyperpageref]{backref}

\usepackage[normalem]{ulem}

\newtheorem{thm}{Theorem}[section]
\newtheorem{lem}[thm]{Lemma}

\newtheorem{prop}[thm]{Proposition}

\theoremstyle{definition}
\newtheorem{rem}[thm]{Remark}
\theoremstyle{remark}

\newcommand{\ds}{\displaystyle}
\newcommand{\norm}[1]{\left\Vert#1\right\Vert}

\newcommand{\abs}[1]{\left\vert#1\right\vert}
\newcommand{\set}[1]{\left\{#1\right\}}
\newcommand{\R}{\mathbb{R}}

\newcommand{\de}{\partial}
\newcommand{\eps}{\varepsilon}

\def\XXint#1#2#3{{\setbox0=\hbox{$#1{#2#3}{\int}$}
    \vcenter{\hbox{$#2#3$}}\kern-.5\wd0}}

\DeclareMathOperator{\dist}{dist}

\newcommand\restr[2]{{
  \left.\kern-\nulldelimiterspace 
  #1 
  \vphantom{ |} 
  \right|_{#2} 
  }}
{\left\{\begin{array}{@{}l@{}}}{\end{array}\right.}
\makeatletter 
\def\@makefnmark{} 
\makeatother 

\title{Asymptotic behavior of the first Robin eigenvalue of nonlinear operators}

\author{Rosa Barbato$^*$, Francesco Della Pietra$^*$, Alba Lia Masiello$^*$ \thanks{
\emph{\textrm{Dipartimento di Matematica e Applicazioni ``R. Caccioppoli'', Universit\`a degli studi di Napoli Federico II, Via Cintia, Complesso Universitario Monte S. Angelo, 80126 Napoli, Italy.\\
 E-mail: rosa.barbato2@unina.it, f.dellapietra@unina.it, albalia.masiello@unina.it}}}}
\date{}

\begin{document}
\maketitle
\begin{abstract}
 \textbf{Abstract.} Let $\Omega$ be a bounded Lipschitz domain of $\R^N$, $N\geq 2$. In this paper, we study the asymptotic behavior  of the first Robin eigenvalue of the $p$-Laplace operator as $\beta$ goes to $0$ and as $\beta$ goes to $+\infty$, deriving sharp asymptotic expansions of the eigenvalue; the expansion in the Dirichlet limit $\beta\to+\infty$ is obtained under the additional assumption that   $\de\Omega$  is of class $C^{1,1}$.


\end{abstract}

\begin{center}
\begin{minipage}{11cm}
\small
\tableofcontents
\end{minipage}
\end{center}

\section{Introduction}

Let $\Omega$ be a bounded, 
open, connected set of $\R^{N}$, $N\ge2$, with Lipschitz boundary. Let $1<p<+\infty$ and let $\beta$ be a real parameter. The first eigenvalue of the $p$-Laplace operator with Robin boundary conditions is 
\begin{equation}
\label{variationalaut}
    \lambda_p(\beta,\Omega)=\min_{w\in W^{1,p}(\Omega)\setminus\{0\}}\dfrac{\displaystyle\int_{\Omega}|\nabla w|^p\;dx+\beta\int_{\de\Omega}|w|^p\:d\mathcal{H}^{N-1}}{\displaystyle\int_{\Omega} |w|^p\;dx},
\end{equation}

where the minimum of \eqref{variationalaut} is achieved by any multiple of the positive function $u_{\beta}\in W^{1,p}(\Omega)$ that satisfies

\begin{equation}
    \label{eigenvalue_problem}
    \begin{cases}
        -\Delta_p u_{\beta}=\lambda_{p}(\beta,\Omega)u_{\beta}^{p-1} & \textrm{in }\Omega\\[1ex]
        \abs{\nabla u_{\beta}}^{p-2}\dfrac{\de u_{\beta}}{\de \nu}+\beta u_{\beta}^{p-1}=0 &\textrm{on }\de\Omega,\\[1ex]
        \norm{u_\beta}_{L^p(\Omega)}=1
    \end{cases}
\end{equation}
where $\nu$ is the outer normal to $\de\Omega$.

We aim to study the asymptotic behavior of the first eigenvalue of the $p$-Laplace operator in the limiting cases as $\beta$ goes to $0$ and $+\infty$.
{It is well known that, as $\beta\to 0$,}  $\lambda_{p}(\beta,\Omega)$ converges to $0$, which corresponds to the first trivial Neumann eigenvalue; {when} $\beta\to+\infty$, $\lambda_{p}(\beta,\Omega)$ converges to $\lambda_p^D(\Omega)$, the first Dirichlet eigenvalue of $-\Delta_{p}$, that is
\begin{equation}\label{vardiri}
  \lambda_p^D(\Omega)=\min_{\varphi\in W_0^{1,p}(\Omega)\setminus\{0\}}\dfrac{\displaystyle\int_{\Omega}\abs{\nabla \varphi}^p\;dx}{\displaystyle\int_{\Omega} \abs{\varphi}^p\;dx}, 
\end{equation}
and any minimizer of \eqref{vardiri} is a multiple of the positive solution to
\begin{equation*}
    \begin{cases}
        -\Delta_p u_{\infty}=\lambda_{p}^D(\Omega)u_{\infty}^{p-1} & \textrm{in }\Omega\\[1ex]
        u_\infty=0 &\textrm{on }\de\Omega,\\[1ex]
        \norm{u_\infty}_{L^p(\Omega)}=1.
    \end{cases}
\end{equation*}

The interest in this kind of problems has grown over the last decades. Many contributions have been given in the linear case $p=2$. When $\beta\to +\infty$, in \cite{fili15b} Filinovskiy  
 investigates the rate of convergence of $\lambda_2(\beta,\Omega)$ to $\lambda_2^D(\Omega)$, proving that 
\begin{equation}
    \label{fili}
     \lambda_2(\beta,\Omega)-\lambda_2^D(\Omega) = -\frac{1}{\beta} \dfrac{\ds\int_{\de\Omega}\abs{\nabla u_\infty}^2\, d\mathcal{H}^{N-1}}{\ds\int_\Omega u_\infty^2\, dx} + o\left(\beta^{-1}\right)
\end{equation}
whenever $\Omega$ is a set of class $C^3$. Under the same regularity assumption, in \cite{fili17} the analysis was extended to higher eigenvalues $\lambda_{2,k}(\beta,\Omega)$, at least when the corresponding Dirichlet eigenvalue $\lambda_{2,k}^D(\Omega)$ is simple. The approach considered in \cite{fili17} makes this assumption essential, since the proof of the asymptotic expansion relies on the differentiability of the Dirichlet eigenvalue and, consequently, of the Robin one for sufficiently large values of $\beta$.

A different approach is used in \cite{BBBT}, where the authors  quantify the eigenvalue gap $\lambda_{2,k}(\beta,\Omega)-\lambda_{2,k}^D(\Omega)$ using a functional analytic approach: this {makes it possible} to treat the simple and multiple eigenvalue case at the same time, but it requires $C^\infty$ regularity of the set $\Omega.$

Subsequently, the strong regularity assumption on $\Omega$ was removed in \cite{ognibene}, where singular perturbation techniques are used to derive the first-order expansion as in \eqref{fili} for bounded Lipschitz domains. By interpreting the Robin eigenvalue problem as a singular perturbation of the corresponding Dirichlet problem for large $\beta$, the author is able to treat both simple and multiple eigenvalues and to extend the result to Lipschitz domains.

In all the aforementioned cases, the linearity of the Laplacian and the fact that the ambient space of the eigenfunctions is a Hilbert space are crucial in determining the asymptotic {expansion} \eqref{fili}. The case $p\neq2$ is more delicate because of the nonlinear structure of the $p$-Laplacian and the possible degeneracy of the operator.

To the best of our knowledge, few results for the $p$-Laplacian are known. For example, when $\beta$ goes to $-\infty$, it was proved in \cite{KP17} (see \cite{papo} for the case $p=2$) that, for a set with $C^{1,1}$ boundary, it holds that
\[ \lambda_{p}(\beta,\Omega) = -(p-1)\abs{\beta}^{\frac{p}{p-1}}- (N-1)H_{max}(\Omega)\abs{\beta}+ o({\abs{\beta}}),
\]
{where $H_{max}(\Omega)$ denotes the maximum of the mean curvature of $\de\Omega$.

Another contribution in the nonlinear setting can be found in \cite{OB}, where the asymptotic behavior of a class of nonlinear variational problems with Robin-type boundary conditions is studied. In place of the eigenvalues, the authors consider the energy 

\[E_\beta=\inf\left\{\frac{1}{p}\int_{\Omega}\abs{\nabla u}^p\, dx+ \frac{\beta}{q}\int_{\de\Omega} |u|^q\, d\mathcal{H}^{N-1}- \int_\Omega fu\, dx \, : \, u\in W^{1,p}(\Omega)\right\}\]
for $p,q\in (1, \infty) $,  on a bounded, Lipschitz domain.
Using variational methods, the authors derive a first-order expansion of $E_\beta$ both as $\beta\to0^+$ and as $\beta\to+\infty$.
They show that, for $\beta\to +\infty$,  the energy
converges to the Dirichlet one with a power-type rate depending only on $q$. An analogous result holds for the Neumann limit when $\beta\to 0$.

Some partial results on the expansion of   
\begin{equation} \label{pgap}
     \lambda_p(\beta,\Omega)-\lambda_p^D(\Omega)
\end{equation}
 can be given exploiting the results proved in \cite{BDP_pota}. The authors prove both upper and lower bounds for $\lambda_p(\beta,\Omega)$, in terms of the Dirichlet spectral datum, the Dirichlet $p$-torsional rigidity of the domain, and the explicit boundary contribution through $\beta$ and the perimeter $P(\Omega)$. The first result is an upper bound for $\lambda_p(\beta,\Omega)$ in terms of the first Dirichlet eigenvalue, that holds whenever
$\Omega$ is open, bounded, with $C^{1,\gamma}$ boundary:

\[
\lambda_p(\beta,\Omega)
\le
\left(
\frac{1}{\lambda_p^D(\Omega)^{\frac{1}{p-1}}}
+
\frac{\overline K}{\big(\beta P(\Omega)\big)^{\frac{1}{p-1}}}
\right)^{-(p-1)}\hspace{-10mm},
\]
where  $\overline K$ is a constant depending only on $p$ and $\Omega$. { This inequality shows that the gap \eqref{pgap} cannot vanish faster than $\beta^{-\frac{1}{p-1}}$.} Further bounds were proved in \cite{BDP_pota}, { such as} the one in terms of the Dirichlet $p$-torsional rigidity  
\[
\lambda_p(\beta,\Omega)
\le
\left[
\frac{T_p(\Omega)}{|\Omega|}
+
\left(\frac{|\Omega|}{\beta P(\Omega)}\right)^{\frac{1}{p-1}}
\right]^{-(p-1)}\hspace{-10mm},
\]
or the lower bound in terms of a Neumann-type constant $\nu_p$ (see Section \ref{limit_zero} for the precise definition),
 \[
\lambda_p(\beta,\Omega)
\ge
\left[
\frac{1}{\nu_p}
+
\left(\frac{|\Omega|}{\beta P(\Omega)}\right)^{\frac{1}{p-1}}
\right]^{-(p-1)}\hspace{-10mm}.
\]

{ In this paper we obtain the exact rate of convergence of $\lambda_{p}(\beta,\Omega)$, both as $\beta\to +\infty$ and $\beta\to 0$.}
Let us now state our first main result. 
\begin{thm}
\label{main_theorem}
    Let $\Omega\subseteq\R^N$ be an open, bounded, connected set with $C^{1,1}$ boundary. Let $\lambda_p(\beta,\Omega)$ and $\lambda_p^D(\Omega)$ be the first { eigenvalues} of the $p$-Laplacian with Robin and Dirichlet boundary conditions, respectively. Let $u_\infty$ be the first positive eigenfunction of the $p$-Laplacian with Dirichlet boundary conditions, normalized such that $\norm{u_\infty}_{L^p(\Omega)}=1$. { Then}
    \begin{equation}
    \label{robin:dirichlet:bound}
        \lambda_p(\beta,\Omega) =\lambda_p^D(\Omega) -\frac{p-1}{\beta^\frac{1}{p-1}}\int  _{\de\Omega} \abs{\nabla u_\infty}^p\, d\mathcal{H}^{N-1} + o\left(\beta^{-\frac{1}{p-1}}\right),
    \end{equation}
    as $\beta \to +\infty$.
\end{thm}

The proof of Theorem \ref{main_theorem} is based on a variational approach and is divided into two  estimates. More precisely, we first establish a sharp upper bound for the gap \eqref{pgap}
by constructing a suitable perturbation of the Dirichlet eigenfunction $u_\infty$, obtained by adding to it a boundary correction term of order $\beta^{-\frac{1}{p-1}}$. This yields the asymptotic upper estimate stated in Proposition \ref{upper_bound}.
The lower bound, proved in Proposition \ref{lower_bound}, is  obtained by choosing a suitable source term dependent on the Robin eigenfunction, comparing the Robin energy with the Dirichlet energy of an auxiliary problem. Using standard global regularity theory for the Dirichlet problem, we identify the exact asymptotic behavior of the boundary terms and derive the lower estimate. 

To conclude the paper, we study the behavior of $\lambda_{p}(\beta,\Omega)$ as $\beta\to 0$. In \cite{BDP_pota}, { the first-order Taylor expansion}

\[\lambda_{p}(\beta,\Omega)=\beta \frac{P(\Omega)}{\abs{\Omega}}+ o(\beta), \quad \text{ as } \beta\to 0,
\]
{ was established;} in Theorem \ref{asymptotic_zero}, we derive the next-order term.
\begin{thm}
\label{asymptotic_zero}
Let $\Omega\subseteq\R^N$ be an open, bounded, connected set with Lipschitz boundary. { Then}
\begin{equation}\label{neumannlim}
\lambda_p(\beta, \Omega) = \beta \frac{P(\Omega)}{|\Omega|} -(p-1) \,{\abs{\beta}^{\frac{p}{p-1}}}\frac{P(\Omega)}{\abs{\Omega}^2}\int_\Omega v\,dx+o\big({\abs{\beta}^{\frac{p}{p-1}}}\big), \qquad \text{ as } \beta \to 0,
\end{equation}
where $v$ is { the unique solution} to
\begin{equation}
    \label{equazionelimite}
 \begin{cases}
    -\Delta_p v=\frac{P(\Omega)}{|\Omega|}& \text{ in } \Omega\\[1ex]
     |\nabla v|^{p-2}\dfrac{\de v}{\de \nu}=-1 & \text{ on } \de\Omega \\[1ex]
     \int_{\de\Omega}v\, d\mathcal{H}^{N-1}=0.
\end{cases}
\end{equation}
\end{thm}

Problem \eqref{equazionelimite} was already identified in \cite{BDP_pota} as the { limit, as $\beta\to0^{+}$, of the Poisson problem with Robin boundary conditions}. More precisely, if $v_\beta$ denotes the solution to
\[
 \begin{cases}
-\Delta_p v_\beta=\dfrac{P(\Omega)}{|\Omega|} & \text{in } \Omega,\\[1ex]
|\nabla v_\beta|^{p-2}\dfrac{\de v_\beta}{\de \nu}+\beta v_\beta^{p-1}=0 & \text{on } \de\Omega,
\end{cases}
\]
then $v_\beta$, normalized by subtracting its mean value on  $\de\Omega$, converges to the solution to \eqref{equazionelimite}.

The structure of the paper is the following. In Section \ref{preliminary} we first recall some preliminary properties of the first Robin eigenvalue (giving the full proofs in Appendix \ref{appendix:properties}), followed by Section \ref{limit_infinity} and Section \ref{limit_zero}, in which we give a detailed study of its asymptotic behavior as $\beta$ goes to $+\infty$ and to $0$ (the latter for both signs of $\beta$), respectively. We point out that the required boundary regularity depends on the regime: the preliminary properties of Section \ref{preliminary} and the Neumann limit $\beta\to0$ (Theorem \ref{asymptotic_zero}) only require   $\de\Omega$  to be Lipschitz, whereas the expansion in the Dirichlet limit $\beta\to+\infty$ (Theorem \ref{main_theorem}) is established for $C^{1,1}$ domains.

\section{Some properties of the first Robin eigenvalue}
\label{preliminary}
{
We start by recalling some well-known properties of the first Robin eigenvalue $\lambda_p(\beta, \Omega)$. For the sake of completeness, we give the full proofs in Appendix \ref{appendix:properties}.

The variational quantity \eqref{variationalaut} is well defined for every $\beta\in\R$. Since the Neumann limit $\beta\to0$ is studied from both sides (see Section \ref{limit_zero}), the results of this section are stated for an arbitrary sign of the parameter; throughout this section $\Omega$ is a bounded, connected, open set with Lipschitz boundary.

We begin recalling a trace interpolation inequality, which is the key tool to handle the boundary term when $\beta<0$.


\begin{lem}\label{traceineq}
 Let $\Omega \subset \R^N$ an open, bounded domain with Lipschitz boundary. Then for every $\varepsilon> 0$ there exists $C(\varepsilon)>0$ for which
   \begin{equation}
 \label{trace:interp}
  \int_{\de \Omega} \abs{w}^p\;d\mathcal{H}^{N-1}\le \varepsilon\int_{\Omega}\abs{\nabla w}^p\;dx+C(\eps)\int_{\Omega}\abs{w}^p\;dx.
   \end{equation}
\end{lem}

\begin{prop}\label{prop:properties}
For every $\beta\in\R$ the infimum in \eqref{variationalaut} is finite and attained. The minimizer of \eqref{variationalaut} is unique up to a multiplicative constant. In particular, the first eigenfunction has constant sign; we denote by $u_\beta$ the positive one normalized by $\norm{u_\beta}_{L^p(\Omega)}=1$.

The map $\beta\mapsto\lambda_p(\beta,\Omega)$ is concave, non-decreasing and continuously differentiable on $\R$, with
\begin{equation}\label{derivata}
\frac{d}{d\beta}\lambda_p(\beta,\Omega)=\int_{\de\Omega}u_\beta^p\,d\mathcal H^{N-1}.
\end{equation}
Moreover, $\beta\mapsto u_\beta$ is continuous from $\R$ to $W^{1,p}(\Omega)$, and $\lambda_p(\beta,\Omega)$ has the same sign as $\beta$, with
\begin{equation}\label{bound:sign}
-\infty<\lambda_p(\beta,\Omega)<\lambda_p^D(\Omega),
\end{equation}
and
\begin{equation*}\label{limiti}
\lim_{\beta\to-\infty}\lambda_p(\beta,\Omega)=-\infty,\qquad \lim_{\beta\to 0}\lambda_p(\beta,\Omega)=0,\qquad \lim_{\beta\to+\infty}\lambda_p(\beta,\Omega)=\lambda_p^D(\Omega).
\end{equation*}

Moreover, $\int_{\de\Omega}u_\beta^p\,d\mathcal H^{N-1}>0$ for every $\beta\in\R$. Consequently, the map $\beta\mapsto\lambda_p(\beta,\Omega)$ is strictly increasing. Finally,
\begin{equation}
\label{eq:strong_ubeta_dirichlet}
u_\beta\longrightarrow u_\infty
\quad\text{strongly in }W^{1,p}(\Omega),
\end{equation}
where $u_{\infty}$ is the first positive normalized Dirichlet eigenfunction.
\end{prop}

}
\section{The case \texorpdfstring{$\beta\to +\infty$}{beta to +infty}}
\label{limit_infinity}
Throughout this section, devoted to proving Theorem \ref{main_theorem}, we set
\[
\gamma:=\frac{1}{p-1}.
\]
We first analyze the behavior of $u_\beta$ on $\de\Omega$. We have the following.
\begin{prop}
\label{boundary:prop}
Let $\Omega\subseteq\R^N$ be an open, bounded, connected set with $C^{1,1}$ boundary.
   Let $\beta>0$, and $u_\beta$ be the first positive eigenfunction of the Robin $p$-Laplacian on $\Omega$, normalized such that
     \[\norm{u_\beta}_{L^p(\Omega)}=1.\]
    Then, there exists a constant $C>0$, independent of  $\beta \ge 1$, such that $\norm{u_\beta}_{L^\infty(\Omega)}\le C$ and
    \begin{equation*}
        \label{boundary:ubeta}
       \norm{u_\beta}  _{L^\infty(\de\Omega)} \le C \beta^{-\gamma}.
    \end{equation*}
\end{prop}
\begin{proof}
We first recall that the family $\{u_\beta\}_{\beta\ge1}$ is uniformly bounded
in $L^\infty(\Omega)$. Indeed, testing the equation with the usual
truncations $(u_\beta-k)^+$, the Robin boundary term has a good sign:
\[
\beta\int_{\de\Omega}u_\beta^{p-1}(u_\beta-k)^+
\,d\mathcal H^{N-1}\ge0.
\]
Thus it can be discarded in the Caccioppoli estimates. Since
$0\le\lambda_p(\beta,\Omega)\le\lambda_p^D(\Omega)$ and
$\|u_\beta\|_{L^p(\Omega)}=1$, the standard Moser iteration (see, e.g.,
\cite{Le2006Eigenvalue}), gives
\[
\|u_\beta\|_{L^\infty(\Omega)}\le C,
\]
where the constant $C$ depends only on $N,p,\Omega$ and on $\lambda_p^D(\Omega)$. In particular, $C$ is independent on $\beta\ge1$.
It follows that
$
\lambda_{p}(\beta,\Omega)\, u_{\beta}^{p-1}
$
is uniformly bounded in $L^{\infty}(\Omega)$.

To establish the precise asymptotic behavior of the boundary trace of $u_\beta$, we  use a barrier argument. Let  $d(x) = \dist(x, \de\Omega)$  be the distance to the boundary of $\Omega$. Since $\Omega$ is of class \(C^{1,1}\), $d$ has the same regularity in a tubular neighborhood $\Omega_\delta = \{x \in \Omega : d(x) < \delta\}$ for a small $\delta > 0$.

Let us consider 
\[
v(x) := C \left(\beta^{-\gamma} + d(x) - \frac{1}{4\delta} d(x)^2\right),
\]
where  $C > 0$ is a constant independent of $\beta$ to be determined. 

We claim that $v$ is a supersolution to the equation \eqref{eigenvalue_problem} in $\Omega_\delta$, for some $C>0$. Indeed, on the boundary  $\de\Omega$, we have $d(x)=0$ and $\nabla d = -\nu$. Hence $\nabla v = -C\nu$, meaning $|\nabla v| = C$ and  $-\frac{\de v}{\de \nu} = C$. So, it holds that
\[
|\nabla v|^{p-2}\left(- \frac{\de v}{\de\nu}\right) = C^{p-1}=\beta v^{p-1},
\]
and  the last equality follows from $v=C\beta^{-\gamma}$ on  $\de\Omega$ and $\gamma(p-1)=1$.
Thus, $v$ satisfies Robin boundary conditions.

 In $\Omega_\delta$, the function $d$ admits finite second derivatives almost everywhere, hence, a direct computation of the $p$-Laplacian gives:
\[
\Delta_p v = C^{p-1} \left(1 - \frac{d}{2\delta} \right)^{p-2} \left[ \left(1 - \frac{d}{2\delta} \right) \Delta d - \frac{p-1}{2\delta} \right].
\]
Since $\Delta d \le K$ in $\Omega_\delta$, we can choose $\delta$ sufficiently small such that $K - \frac{p-1}{2\delta} \le -c_0 < 0$, and we get
\[
-\Delta_p v \ge \tilde c_0 C^{p-1}
\]
for some positive constant $\tilde c_0$.

  Thus, since $ \lambda_{p}(\beta,\Omega)\, u_{\beta}^{p-1} $ is uniformly bounded in $L^{\infty}(\Omega)$, we can choose $C$ large enough such that $-\Delta_p v \ge \lambda_p(\beta,\Omega) u_\beta^{p-1}$.

 On the inner boundary $d(x) = \delta$, we have $v = C(\beta^{-\gamma} + \frac{3}{4}\delta)$. Since $\delta$ is fixed, by potentially enlarging $C$ further, we can guarantee that $v \ge \|u_\beta\|_{L^\infty(\Omega)} \ge u_\beta$ { on $\{d=\delta\}$}.

Then, it holds that
\[
u_\beta(x) \le v(x) \quad\text{in }\Omega_\delta.
\]
{ Indeed, subtracting the weak formulations of $u_\beta$ and $v$ in $\Omega_\delta$ tested with $\varphi = (u_\beta - v)^+$ (which vanishes on $\{d=\delta\}$), and using the boundary conditions on  $\de\Omega$, we obtain
\begin{multline*}
\int_{\{u_\beta > v\}} (|\nabla u_\beta|^{p-2}\nabla u_\beta - |\nabla v|^{p-2}\nabla v) \cdot (\nabla u_\beta - \nabla v) \, dx \\+  \beta \int_{\de\Omega \cap \{u_\beta > v\}} (u_\beta^{p-1} - v^{p-1}) (u_\beta - v) \, d\mathcal{H}^{N-1}
 \le 0.
\end{multline*}
Both terms on the left-hand side are nonnegative, by the monotonicity of the maps $\xi\mapsto\abs{\xi}^{p-2}\xi$ and $t\mapsto t^{p-1}$; hence both vanish. In particular $\nabla(u_\beta-v)^+=0$ a.e. in $\Omega_\delta$, so that $(u_\beta-v)^+$ is constant on each connected component of $\Omega_\delta$; since it vanishes on $\{d=\delta\}$, it is identically zero, that is $u_\beta \le v$ in $\Omega_\delta$.}

 Evaluating this inequality where $d(x)=0$ yields $u_\beta \le C \beta^{-\gamma}$ on  $\de\Omega$. Since $u_\beta > 0$, we have proven that 
\[
\|u_\beta\| _{L^\infty(\de\Omega)} = O(\beta^{-\gamma}).
\]
 \end{proof}

As anticipated in the introduction, the proof of Proposition \ref{upper_bound} 
relies on the construction of suitable test functions for the variational characterizations \eqref{variationalaut} and \eqref{vardiri}. In both cases, a crucial role is played by the function $w$ introduced in the following proposition.

\begin{prop}
\label{trace}
Let $\Omega\subseteq\mathbb R^N$ be an open, bounded, connected set with
$C^{1,1}$ boundary, and let $u_\infty$ be the positive first eigenfunction of
the Dirichlet $p$-Laplacian, normalized by $\|u_\infty\|_{L^p(\Omega)}=1.$ Then, there exists a function
$w\in W^{1,p}(\Omega)\cap L^\infty(\Omega)$ such that
\[
w=|\nabla u_\infty|
\qquad\text{on }  \de \Omega.
\]
\end{prop}

\begin{proof}
By the boundary $C^{1,\alpha}$-regularity of $u_\infty$ and Hopf's lemma
(see \cite{DiB,Lieberman,V}), and the fact that   $\de\Omega=\{u_{\infty}=0\}$ , we have 
\[|\nabla u_\infty| =-  \frac{\de u_\infty}{\de\nu} \ge c_1> 0,  \text{ on }   \de \Omega.
\]
In particular, there exists a tubular neighborhood $\Omega_\delta$ of   $\de \Omega$   in which $|\nabla u_\infty|\ge c_1/2$.

Hence in $\Omega_\delta$, the equation is uniformly elliptic and non-degenerate. Moreover, { since}   $\de\Omega$  is $C^{1,1}$, the standard $L^q$ elliptic regularity theory guarantees that $u_\infty \in W^{2,q}(\Omega_\delta)$ for any $q<\infty$ (see for instance \cite[Theorem 6.5]{ladyzhenskaya}).

In particular, taking $q \ge p$, the trace $|\nabla u_\infty|$ belongs to   $W^{1-1/p,p}(\de\Omega)$ , which ensures the existence of a global extension $w \in W^{1,p}(\Omega)$ of $|\nabla u_\infty|$ (see \cite[Chapter 2, Theorem 5.7]{N}); since   $|\nabla u_\infty|\in L^\infty(\de\Omega)$ , replacing $w$ with its truncation at the level   $\norm{\nabla u_\infty}_{L^\infty(\de\Omega)}$  if necessary, we may also assume $w\in L^\infty(\Omega)$.

Thus
\[
w\in W^{1,p}(\Omega)\cap L^\infty(\Omega),
\qquad
w=|\nabla u_\infty|
\quad\text{on }  \de \Omega.
\]

\end{proof}

As stated in the introduction, the proof of Theorem \ref{main_theorem} is a direct consequence of the following two steps: we prove that the gap $\lambda_p(\beta,\Omega)-\lambda_p^D(\Omega)$ is asymptotically bounded from above and from below  by the right-hand side of \eqref{robin:dirichlet:bound} in Propositions \ref{upper_bound} and \ref{lower_bound} respectively.

\begin{prop}
\label{upper_bound}
    Let $\Omega\subseteq\R^N$ be an open, bounded, connected set with $C^{1,1}$ boundary.  Let $\lambda_p(\beta,\Omega)$ and $\lambda_p^D(\Omega)$ be the first eigenvalues of the $p$-Laplacian with Robin and Dirichlet boundary conditions, respectively. Let $u_\infty$ be the first positive eigenfunction of the $p$-Laplacian with Dirichlet boundary conditions, normalized such that $\norm{u_\infty}_{L^p(\Omega)}=1$. { Then,}
    \begin{equation*}
 \lambda_p(\beta,\Omega)-\lambda_p^D(\Omega) \le -\frac{p-1}{\beta^\frac{1}{p-1}}\int  _{\de\Omega} \abs{\nabla u_\infty}^p\, d\mathcal{H}^{N-1} + o\left(\beta^{-\frac{1}{p-1}}\right),\qquad \beta\to +\infty.
    \end{equation*}
\end{prop}
\begin{proof}
Let $u_\beta$ and $u_\infty$ be the first positive normalized eigenfunctions with Robin and Dirichlet boundary conditions.

We will prove \eqref{robin:dirichlet:bound} by establishing an upper  bound for $\lambda_p(\beta,\Omega)-\lambda_p^D(\Omega)$. We construct a test function $\psi_\beta$ as a perturbation of the positive Dirichlet eigenfunction $u_\infty$:
    \[
    \psi_\beta = u_\infty + \beta^{-\gamma} w,
    \]
    where $w \in W^{1,p}(\Omega)$ is  the extension of $\abs{\nabla u_\infty}$ defined in Proposition \ref{trace}.

    We expand the denominator of the Rayleigh quotient \eqref{variationalaut} evaluated at $\psi_\beta$
by using a Taylor expansion for $f(t)=|t|^p$. As $\beta \to +\infty$,
    \begin{align}
        D[\psi_\beta] &= \int_\Omega |u_\infty + \beta^{-\gamma} w|^p dx \notag \\
        &= \int_\Omega u_\infty^p dx + p\beta^{-\gamma} \int_\Omega u_\infty^{p-1} w \, dx + o(\beta^{-\gamma}) \label{denominatore:espansione} \\
        &= 1 + p\beta^{-\gamma} C_w + o(\beta^{-\gamma}), \notag
    \end{align}
    where $C_w = \int_\Omega u_\infty^{p-1} w \, dx$.\\

  The numerator in the Rayleigh quotient \eqref{variationalaut} is
    \[
    N[\psi_\beta] = \int_\Omega |\nabla u_\infty + \beta^{-\gamma} \nabla w|^p dx + \beta \int  _{\de\Omega}  |u_\infty + \beta^{-\gamma} w|^p d \mathcal H^{N-1}.
    \]
The first term can be expanded by using the vector inequality of Lemma
\ref{lem:algebraic_taylor}, together with the boundedness of $\nabla u_\infty$, obtaining
    \begin{align*}
           \int_\Omega |\nabla u_\infty + \beta^{-\gamma} \nabla w|^p dx &= \int_\Omega |\nabla u_\infty|^p dx + p\beta^{-\gamma} \int_\Omega |\nabla u_\infty|^{p-2}\nabla u_\infty \cdot \nabla w \, dx + o(\beta^{-\gamma}) \\ &= \lambda_p^D(\Omega) + p\beta^{-\gamma} \int_\Omega |\nabla u_\infty|^{p-2}\nabla u_\infty \cdot \nabla w \, dx + o(\beta^{-\gamma}).
    \end{align*}
By integration by parts formula, recalling that $|\nabla u_\infty|^{p-2}\nabla u_\infty$ has a well defined normal trace on $\de \Omega$, that $w=|\nabla u_{\infty}|$ on $\de\Omega$ and that $u_\infty$ is a Dirichlet eigenfunction, we have
    \begin{equation*}
    \int_\Omega |\nabla u_\infty|^{p-2}\nabla u_\infty \cdot \nabla w \, dx =-\int  _{\de \Omega} \abs{\nabla u_{\infty}}^p\,d\mathcal{H}^{N-1}+\lambda_p^D(\Omega) \int_\Omega u_\infty^{p-1} w \, dx.
    \end{equation*}
        On the other hand, using   $u_\infty|_{\de\Omega}=0$  and the relation $\beta (\beta^{-\gamma})^p = \beta^{-\gamma}$, the boundary term in $N[\psi_\beta]$ becomes
    \[
    \beta \int  _{\de\Omega}  |u_\infty + \beta^{-\gamma} w|^p d \mathcal H^{N-1} = \beta^{-\gamma} \int  _{\de\Omega}  |w|^p d \mathcal H^{N-1}= \beta^{-\gamma} \int  _{\de\Omega}  \abs{\nabla u_\infty}^p d \mathcal H^{N-1}.
    \]
Then    
    \begin{align*}
        N[\psi_\beta] &= \lambda_p^D(\Omega)  \\ &\phantom{=}+ p\beta^{-\gamma} \left[ -\int  _{\de\Omega}  \abs{\nabla u_\infty}^p\, d \mathcal H^{N-1} + \lambda_p^D(\Omega) C_w \right] + \beta^{-\gamma} \int  _{\de\Omega}  |w|^p d \mathcal H^{N-1} + o(\beta^{-\gamma}) \\
        &= \lambda_p^D(\Omega) (1 + p\beta^{-\gamma} C_w) -(p-1) \beta^{-\gamma} \int  _{\de\Omega}  |\nabla u_\infty|^p d \mathcal H^{N-1}  + o(\beta^{-\gamma})\\
        &= \lambda_p^D(\Omega)D[\psi_\beta] - (p-1)\beta^{-\gamma} \int  _{\de\Omega}  |\nabla u_\infty|^p  d \mathcal H^{N-1}  + o(\beta^{-\gamma}),
    \end{align*}
    where in the last equality we used \eqref{denominatore:espansione}. Hence, computing the Rayleigh quotient, we get
    \begin{align*}
        \lambda_p(\beta,\Omega)  &\le \frac{ \displaystyle N[\psi_\beta] }{ \displaystyle D[\psi_\beta] } \\ &= \lambda_p^D(\Omega) -   \frac{ \displaystyle \beta^{-\gamma} (p-1)\int_{\de\Omega} \abs{\nabla u_{\infty}}^p d \mathcal H^{N-1} + o(\beta^{-\gamma}) }{ \displaystyle D[\psi_\beta] }  \\
 &= \lambda_p^D(\Omega) -\beta^{-\gamma} (p-1)\int  _{\de\Omega}   \abs{\nabla u_{\infty}}^p  d \mathcal H^{N-1} + o(\beta^{-\gamma}),
    \end{align*}
    where  we used \eqref{denominatore:espansione} again. The proof is completed.
\end{proof}

\begin{prop}
\label{lower_bound}
Under the assumptions of Proposition \ref{upper_bound}, it holds that
    \begin{equation}
    \label{robin:dirichlet:bound2}
        \lambda_p(\beta,\Omega)-\lambda_p^D(\Omega) \ge -\frac{p-1}{\beta^\frac{1}{p-1}}\int  _{\de\Omega} \abs{\nabla u_\infty}^p\, d\mathcal{H}^{N-1} + o\left(\beta^{-\frac{1}{p-1}}\right),\quad\beta\to+\infty.
    \end{equation}
 \end{prop}
 \begin{proof}
    Let us set 
   \[
    \lambda_\beta:=\lambda_p(\beta,\Omega), \qquad f_\beta:=u_\beta^{p-1}, \qquad p':=\frac{p}{p-1}. 
    \]
    The  normalization of $u_\beta$  gives 
    \[ 
    	\|f_\beta\|_{L^{p'}(\Omega)}=1.
    \]
  The function $U_\beta := \lambda_\beta^{-\gamma} u_\beta$ solves
   \begin{equation*}
   \begin{cases}
   -\Delta_p U_\beta = f_\beta & \text{in } \Omega,\\
   |\nabla U_\beta|^{p-2}\frac{\de U_\beta}{\de\nu} + \beta U_\beta^{p-1} = 0 & \text{on } \de\Omega.
   \end{cases}
   \end{equation*}
   As a consequence, $U_\beta$ is the minimizer for the energy functional
   \[
   E_\beta(f_\beta) := \inf_{v\in W^{1,p}(\Omega)} \left\{ \frac{1}{p}\int_\Omega |\nabla v|^p\,dx + \frac{\beta}{p}\int_{\de\Omega} |v|^p\,d\mathcal{H}^{N-1} - \int_\Omega f_\beta v\,dx \right\},
   \]
   and  testing the equation with $U_\beta$ gives
   \begin{equation}
   \label{eq:E_beta_eval}
   E_\beta(f_\beta) = -\frac{1}{p'} \lambda_\beta^{-\gamma}.
   \end{equation}
   On the other hand, let us define the analogous Dirichlet energy
   \begin{equation}
   \label{eq:inftybeta}
   E_\infty(f_\beta) := \inf_{\varphi\in W^{1,p}_0(\Omega)} \left\{ \frac{1}{p}\int_\Omega |\nabla \varphi|^p\,dx - \int_\Omega f_\beta \varphi\,dx \right\}.
   \end{equation}
   For any $\varphi \in W^{1,p}_0(\Omega)$, by definition of $\lambda_{p}^{D}(\Omega)$ it holds that
    \[
   \int_\Omega f_\beta \varphi\,dx \le \|f_\beta\|_{L^{p'}(\Omega)} \|\varphi\|_{L^p(\Omega)} \le \lambda_p^D(\Omega)^{-1/p} \|\nabla \varphi\|_{L^p(\Omega)}.
   \]
   Minimizing $\frac{1}{p} X^{p}-\lambda_{p}^{D}(\Omega)^{-1/p}X$ with respect to $X = \|\nabla \varphi\|_{L^p(\Omega)}$, we obtain
   \begin{equation}
   \label{eq:E_infty_eval}
   E_\infty(f_\beta) \ge -\frac{1}{p'} \lambda_p^D(\Omega)^{-\gamma}.
   \end{equation}

   Let  $V_\beta \in W^{1,p}_0(\Omega)$ be the unique minimizer of \eqref{eq:inftybeta}, which solves the auxiliary Dirichlet problem
   \begin{equation*}
   \begin{cases}
   -\Delta_p V_\beta = f_\beta & \text{in } \Omega,\\
   V_\beta = 0 & \text{on } \de\Omega,
   \end{cases}
   \end{equation*}
 By Proposition \ref{boundary:prop}, $ \|f_\beta\|_{L^\infty(\Omega)} \le C$ as $\beta$ is large. This implies
\[
\sup_{\beta\ge 1}
\|V_\beta\|_{L^\infty(\Omega)}
<+\infty.
\]
 The global boundary regularity estimates for the Dirichlet problem 
(\cite[Theorem 1]{Lieberman}) therefore yield the existence of $\alpha\in(0,1)$ and  $C>0$, both independent of $\beta\geq1$, such that
\begin{equation}
\label{eq:uniform_regular_Vbeta}
\|V_\beta\|_{C^{1,\alpha}(\overline\Omega)}
\leq C.
\end{equation}

 By the convexity of  $\xi\mapsto|\xi|^p/p$, for every
$v\in W^{1,p}(\Omega)$ we have
\begin{align*}
\frac{1}{p}\int_\Omega|\nabla v |^p\, dx-\int _\Omega f_\beta  v\,dx
 & \geq
\frac{1}{p}\int_\Omega |\nabla V_\beta|^p\,dx
-\int _\Omega f_\beta  V_\beta \,dx
 \\
& \quad
+\int_\Omega
 |\nabla V_\beta|^{p-2}\nabla V_\beta
\cdot\nabla(v-V_\beta)\,dx
- \int_\Omega  f_\beta (v-V_\beta)\,dx  \\
&= E_\infty(f_\beta)
+
\int  _{\de\Omega}
|\nabla V _\beta  |^{p-2}
\frac{\de V_\beta}{\de\nu} \,
 v\,d\mathcal H^{N-1},
\end{align*}
where the last equality follows from integration by parts.

 Moreover, the strong maximum principle and Hopf Lemma imply that $V_\beta>0$ in $\Omega$ and  $\frac{\de V_\beta}{\de\nu}<0$ on $\de\Omega$. Taking $v = U_\beta$ yields
   \[
   E_\beta(f_\beta) \ge E_\infty(f_\beta) + \int_{\de\Omega} \left( \frac{\beta}{p}U_\beta^p - |\nabla V_\beta|^{p-1}U_\beta \right)d\mathcal{H}^{N-1}.
   \]
 For every $a,t\geq0$, we have
\[
\frac{\beta}{p}t^p-a^{p-1}t
\geq
-\frac{1}{p'}\beta^{-\gamma}a^p,
\]
with  equality for $t=\beta^{-\gamma}a$. For $a=|\nabla V_\beta|$ and $t=U_\beta$, we obtain   \begin{equation*}
   E_\beta(f_\beta) \ge E_\infty(f_\beta) - \frac{1}{p'\beta^\gamma} \int_{\de\Omega} |\nabla V_\beta|^p\,d\mathcal{H}^{N-1}.
   \end{equation*}
    Combining this with \eqref{eq:E_beta_eval} and \eqref{eq:E_infty_eval} yields
   \begin{equation}
   \label{eq:eigenvalue_comparison_dual}
   \lambda_\beta^{-\gamma} \le \lambda_p^D(\Omega)^{-\gamma} + \beta^{-\gamma} \int_{\de\Omega} |\nabla V_\beta|^p\,d\mathcal{H}^{N-1}.
   \end{equation}

 It remains to identify the limit of $V_\beta$ as
$\beta\to+\infty$.  By strong convergence \eqref{eq:strong_ubeta_dirichlet}, it holds
$f_\beta=u_\beta^{p-1}
\longrightarrow
u_\infty^{p-1}
\text{ strongly in }L^{p'}(\Omega).$
The unique solution of the limiting Dirichlet problem is
\[
V_\infty
=
\lambda_p^D(\Omega)^{-\gamma}u_\infty.
\]

 By the uniform estimate \eqref{eq:uniform_regular_Vbeta}, every
sequence $\beta_n\to+\infty$ admits a subsequence and a function $V\in C^1(\overline\Omega)$ such that
\[
V_{\beta_n}\longrightarrow V
\quad\text{in }C^1(\overline\Omega).
\]
 Passing to the limit in the weak formulation of the equation for
$V_{\beta_n}$, we find
\[
\begin{cases}
-\Delta_pV=u_\infty^{p-1}&\text{in }\Omega,\\
V=0&\text{on }\de\Omega.
\end{cases}
\]
 By uniqueness,
\[
V=V_\infty
=
\lambda_p^D(\Omega)^{-\gamma}u_\infty.
\]
 Every convergent subsequence has the same limit; hence the whole
family satisfies
\begin{equation}
\label{eq:C1_convergence_Vbeta}
V_\beta
\longrightarrow
\lambda_p^D(\Omega)^{-\gamma}u_\infty
\quad\text{in }C^1(\overline\Omega).
\end{equation}
     Consequently,
   \begin{equation}
   \label{eq:V_beta_trace_limit}
   \int_{\de\Omega} |\nabla V_\beta|^p\,d\mathcal{H}^{N-1} \longrightarrow \lambda_p^D(\Omega)^{-p\gamma} \int_{\de\Omega} |\nabla u_\infty|^p\,d\mathcal{H}^{N-1}.
   \end{equation} 
    Inserting \eqref{eq:V_beta_trace_limit} into \eqref{eq:eigenvalue_comparison_dual}, we obtain
   \[
    \lambda_\beta ^{-\gamma} \le  \lambda_p^D(\Omega)^{-\gamma}  +  \beta^{-\gamma} \lambda_p^D(\Omega)^{-p\gamma} \int  _{\de\Omega} |\nabla  u _\infty|^p\, d\mathcal{H}^{N-1} + o(\beta^{-\gamma}).
   \]
    Since $p\gamma = \gamma + 1$, a Taylor expansion of the map $t \mapsto t^{-1/\gamma}$ around $t = \lambda_p^D(\Omega)^{-\gamma}$ yields
   \begin{align*}
   \lambda_\beta &\ge \left( \lambda_p^D(\Omega)^{-\gamma} + \beta^{-\gamma} \lambda_p^D(\Omega)^{-\gamma-1} \int_{\de\Omega} |\nabla u_\infty|^p\,d\mathcal{H}^{N-1} + o(\beta^{-\gamma}) \right)^{-1/\gamma} \\
   &= \lambda_p^D(\Omega) - \frac{1}{\gamma} \beta^{-\gamma} \int_{\de\Omega} |\nabla u_\infty|^p\,d\mathcal{H}^{N-1} + o(\beta^{-\gamma}),
   \end{align*}
    which is the desired lower bound, since $\frac{1}{\gamma} = p-1$.
\end{proof}

\begin{proof}[Proof of Theorem \ref{main_theorem}]
The expansion \eqref{robin:dirichlet:bound} follows by combining Propositions \ref{upper_bound} and \ref{lower_bound}.
\end{proof}

\section{The case \texorpdfstring{$\beta\to 0$}{beta to 0}}
\label{limit_zero}
We now turn to the expansion of the Robin eigenvalue as $\beta$ goes to $0$.

\begin{proof}[Proof of  Theorem \ref{asymptotic_zero}]
Let $u_\beta$ be the normalized positive eigenfunction associated to $\lambda_p(\beta, \Omega)$. Since $\lambda_p(\beta, \Omega) \to 0$ as $\beta \to 0$, testing the eigenvalue equation with $u_\beta$ gives $\int_\Omega |\nabla u_\beta|^p \, dx \to 0$ (for $\beta<0$ the boundary term is absorbed by the trace inequality of Lemma \ref{traceineq}). Hence $u_\beta$ converges strongly in $W^{1,p}(\Omega)$ to a constant $c$, and the normalization implies $c = |\Omega|^{-1/p}$; in particular $u_\beta \to |\Omega|^{-1/p}$ strongly in $L^p(\Omega)$ and in   $L^p(\de\Omega)$. In what follows the limit $\beta\to0$ is always taken along a fixed sign of $\beta$, so that $\mathrm{sign}(\beta)$ is constant; the cases $\beta\to0^+$ and $\beta\to0^-$ are treated simultaneously.

Let us consider
\[ m_\beta := \frac{1}{P(\Omega)} \int  _{\de\Omega}  u_\beta \, d\mathcal H^{N-1} \]
the boundary mean of $u_\beta$, and set the test function $\psi_\beta = u_\beta - m_\beta$. By construction,   $\int_{\de\Omega} \psi_\beta \, d\mathcal H^{N-1} = 0$ , and since $u_\beta \to |\Omega|^{-1/p}$ in   $L^1(\de\Omega)$ , we have $m_\beta \to |\Omega|^{-1/p}$ as $\beta \to 0$.

We claim that $\|\psi_\beta\|_{W^{1,p}(\Omega)} = O(|\beta|^\gamma)$ as $\beta\to 0$, with $\gamma = \frac{1}{p-1}$. 
Testing the weak formulation of the eigenvalue problem with $\psi_\beta$, we get
\begin{equation}
\label{eq:pde_test_psi}
\int_\Omega |\nabla \psi_\beta|^p \, dx + \beta \int  _{\de\Omega}  u_\beta^{p-1} \psi_\beta \, d\mathcal H^{N-1} = \lambda_p(\beta, \Omega) \int_\Omega u_\beta^{p-1} \psi_\beta \, dx.
\end{equation}
The integrals $\int_\Omega u_\beta^{p-1}\psi_\beta\,dx$ and   $\int_{\de\Omega}u_\beta^{p-1}\psi_\beta\,d\mathcal H^{N-1}$  are controlled by $\|\nabla\psi_\beta\|_{L^p(\Omega)}$. Indeed, since $\|u_\beta\|_{L^p(\Omega)}=1$ and $\|u_\beta\|_{L^p(\de\Omega)}$ is bounded, H\"older's inequality, the trace theorem and the Poincar\'e inequality (recall   $\int_{\de\Omega}\psi_\beta\,d\mathcal H^{N-1}=0$ ) give
\[
\left|\int_\Omega u_\beta^{p-1}\psi_\beta\,dx\right|\le \|u_\beta\|_{L^p(\Omega)}^{p-1}\|\psi_\beta\|_{L^p(\Omega)}\le C\|\nabla\psi_\beta\|_{L^p(\Omega)},
\]
and
\[
\left|\int  _{\de\Omega}  u_\beta^{p-1}\psi_\beta\,d\mathcal H^{N-1}\right|\le \|u_\beta\|_{L^p(\de\Omega)}^{p-1}\|\psi_\beta\|  _{L^p(\de\Omega)} \le C\|\nabla\psi_\beta\|_{L^p(\Omega)}.
\]
Hence, from \eqref{eq:pde_test_psi} and the bound $\abs{\lambda_p(\beta,\Omega)}\le C\abs{\beta}$ for $\beta$ small, which follows from Proposition \ref{prop:properties} (the map $\lambda_p(\cdot,\Omega)$ is of class $C^1(\R)$ and vanishes at $\beta=0$), we obtain (or, equivalently, using the trace inequality  and the definition of $\lambda_p(\beta,\Omega)$)
\[
\|\nabla\psi_\beta\|_{L^p(\Omega)}^p=\left|\lambda_p(\beta,\Omega)\int_\Omega u_\beta^{p-1}\psi_\beta\,dx-\beta\int  _{\de\Omega} u_\beta^{p-1}\psi_\beta\,d\mathcal H^{N-1}\right|\le C|\beta|\,\|\nabla\psi_\beta\|_{L^p(\Omega)},
\]
which yields $\|\nabla\psi_\beta\|_{L^p(\Omega)}\le C|\beta|^{\frac{1}{p-1}}=C|\beta|^\gamma$ and, again by the Poincar\'e inequality, $\|\psi_\beta\|_{W^{1,p}(\Omega)}\le C|\beta|^\gamma$.

Rescaling, the function $w_\beta :=|\beta|^{-\gamma}\psi_\beta$ has zero boundary mean and is uniformly bounded in $W^{1,p}(\Omega)$; in particular $\{\int_\Omega w_\beta\,dx\}_\beta$ is bounded.

We now extract the second-order asymptotic expansion of the Rayleigh quotient as $\beta\to 0$. Being $u_\beta = m_\beta + |\beta|^\gamma w_\beta$, the normalization condition $\int_\Omega u_\beta^p \, dx = 1$ reads 
\[ 1 = \int_\Omega (m_\beta + |\beta|^\gamma w_\beta)^p \, dx = m_\beta^p |\Omega| + p m_\beta^{p-1} |\beta|^\gamma \int_\Omega w_\beta \, dx + o(|\beta|^\gamma). \]
Hence 
\begin{equation}
\label{eq:m_beta_expansion}
m_\beta^p = \frac{1}{|\Omega|} - |\beta|^\gamma \frac{p}{|\Omega|} m_\beta^{p-1} \int_\Omega w_\beta \, dx + o(|\beta|^\gamma).
\end{equation}
Substituting into the Rayleigh quotient, we get
\begin{align*}
\lambda_p(\beta, \Omega) &= \int_\Omega |\nabla u_\beta|^p \, dx + \beta \int  _{\de\Omega}  u_\beta^p \, d\mathcal H^{N-1} \\
&= |\beta|^{p\gamma} \int_\Omega |\nabla w_\beta|^p \, dx + \beta \int  _{\de\Omega}  (m_\beta + |\beta|^\gamma w_\beta)^p \, d\mathcal H^{N-1}.
\end{align*}
Expanding the boundary integral and recalling that   $\int_{\de\Omega} w_\beta \, d\mathcal H^{N-1} = 0$ , the linear term in $w_\beta$ vanishes and, by the same Taylor estimate as above (now on   $\de\Omega$ , where $\{w_\beta\}$ is bounded in   $L^p(\de\Omega)$),
 \[ \int  _{\de\Omega}  (m_\beta + |\beta|^\gamma w_\beta)^p \, d\mathcal H^{N-1} = P(\Omega) m_\beta^p + o(|\beta|^\gamma). \]
Substituting \eqref{eq:m_beta_expansion} into the above expression, we obtain:
\begin{multline}
\label{eq:lambda_expansion_vbeta}
\lambda_p(\beta, \Omega) = \beta \frac{P(\Omega)}{|\Omega|} + |\beta|^{\frac{p}{p-1}} \left( \int_\Omega |\nabla w_\beta|^p \, dx - \text{sign}(\beta) p \frac{P(\Omega)}{|\Omega|} m_\beta^{p-1} \int_\Omega w_\beta \, dx \right)  \\ + o\big(|\beta|^{\frac{p}{p-1}}\big).
\end{multline}
where we used the identity $\beta |\beta|^\gamma = \text{sign}(\beta) |\beta|^{p\gamma}$, as well as that $p\gamma = \frac{p}{p-1}$.

We aim to determine the limit of the expression given by the parentheses in \eqref{eq:lambda_expansion_vbeta}. From now on, we work in the subspace
\[
X:=\left\{\varphi\in W^{1,p}(\Omega):
\int  _{\de\Omega} \varphi\,d\mathcal H^{N-1}=0
\right\}.
\]
It holds that $w_{\beta}\in X$. Recall that \eqref{equazionelimite} admits a unique solution in $X$; equivalently, $v$ is the unique minimizer, on $X$, of the corresponding strictly convex and coercive functional. Set
\[
V:=\operatorname{sgn}(\beta)|\Omega|^{-\frac{1}{p}}v.
\]
Then $V\in X$ and it satisfies
\begin{multline}
\label{eq:solV}
\int_\Omega|\nabla V|^{p-2}\nabla V\cdot\nabla\varphi\,dx= \\ = L(\varphi):=\text{sign}(\beta)|\Omega|^{-\frac{p-1}{p}}\left(\frac{P(\Omega)}{|\Omega|}\int_\Omega\varphi\,dx-\int  _{\de\Omega} \varphi\,d\mathcal H^{N-1}\right)
\end{multline}
for every $\varphi\in W^{1,p}(\Omega)$. 

 Choosing $\varphi=V$
\begin{equation}\label{eq:energyV}
\int_\Omega|\nabla V|^p\,dx=L(V)=\frac{P(\Omega)}{|\Omega|^2}\int_\Omega v\,dx.
\end{equation}
\emph{Lower bound.} By the convexity of $\xi\mapsto|\xi|^p$ and by choosing $\varphi=w_\beta$ and $\varphi=V$ in \eqref{eq:solV},
\[
\int_\Omega|\nabla w_\beta|^p\,dx\ge\int_\Omega|\nabla V|^p\,dx+p\int_\Omega|\nabla V|^{p-2}\nabla V\cdot(\nabla w_\beta-\nabla V)\,dx=p\,L(w_\beta)-(p-1)L(V).
\]
Since   $\int_{\de\Omega}w_\beta\,d\mathcal H^{N-1}=0$ , we have $L(w_\beta)=\text{sign}(\beta)|\Omega|^{-\frac{p-1}{p}}\frac{P(\Omega)}{|\Omega|}\int_\Omega w_\beta\,dx$, whence
\[
p\,L(w_\beta)-\text{sign}(\beta)\,p\,\frac{P(\Omega)}{|\Omega|}\,m_\beta^{p-1}\int_\Omega w_\beta\,dx=\text{sign}(\beta)\,p\,\frac{P(\Omega)}{|\Omega|}\Big(|\Omega|^{-\frac{p-1}{p}}-m_\beta^{p-1}\Big)\int_\Omega w_\beta\,dx\longrightarrow0,
\]
because $m_\beta\to|\Omega|^{-1/p}$ and $\{\int_\Omega w_\beta\,dx\}$ is bounded. Inserting the lower bound for $\int_\Omega|\nabla w_\beta|^p\,dx$ into the bracket of \eqref{eq:lambda_expansion_vbeta} and recalling \eqref{eq:energyV}, we obtain
\[
\liminf_{\beta\to0}\frac{\lambda_p(\beta,\Omega)-\beta\frac{P(\Omega)}{|\Omega|}}{|\beta|^{\frac{p}{p-1}}}\ge-(p-1)L(V)=-(p-1)\frac{P(\Omega)}{|\Omega|^2}\int_\Omega v\,dx.
\]

\emph{Upper bound.} We test the Rayleigh quotient \eqref{variationalaut} with
\[
\Phi_\beta:=|\Omega|^{-\frac{1}{p}}+|\beta|^\gamma V=|\Omega|^{-\frac{1}{p}}\big(1+\text{sign}(\beta)|\beta|^\gamma v\big)\in W^{1,p}(\Omega).
\]
Using   $\int_{\de\Omega}V\,d\mathcal H^{N-1}=0$ , we get
\[
\int_\Omega|\nabla\Phi_\beta|^p\,dx=|\beta|^{\frac{p}{p-1}}\int_\Omega|\nabla V|^p\,dx,\qquad
\int_\Omega{ \abs{\Phi_\beta}^p}\,dx=1+p\,|\Omega|^{-\frac{p-1}{p}}|\beta|^\gamma\int_\Omega V\,dx+o(|\beta|^\gamma),
\]
\[
\beta\int  _{\de\Omega}{ \abs{\Phi_\beta}^p} \,d\mathcal H^{N-1}=\beta\frac{P(\Omega)}{|\Omega|}+o\big(|\beta|^{\frac{p}{p-1}}\big).
\]
Hence, using \eqref{eq:energyV},
\begin{multline*}
\lambda_p(\beta,\Omega)\le  \frac{\displaystyle\int_\Omega|\nabla\Phi_\beta|^p\,dx+\beta\int_{\de\Omega}\abs{\Phi_\beta}^p\,d\mathcal H^{N-1}}{\displaystyle\int_\Omega\abs{\Phi_\beta}^p\,dx}
 \\
=\beta\frac{P(\Omega)}{|\Omega|}-(p-1)|\beta|^{\frac{p}{p-1}}\frac{P(\Omega)}{|\Omega|^2}\int_\Omega v\,dx\;+o\big(|\beta|^{\frac{p}{p-1}}\big),
\end{multline*}
so that $\displaystyle\limsup_{\beta\to0}\Big(\lambda_p(\beta,\Omega)-\beta\frac{P(\Omega)}{|\Omega|}\Big)|\beta|^{-\frac{p}{p-1}}\le-(p-1)\frac{P(\Omega)}{|\Omega|^2}\int_\Omega v\,dx.$

Combining the two bounds, the limit exists and
\[
\lim_{\beta\to0}\frac{\lambda_p(\beta,\Omega)-\beta\frac{P(\Omega)}{|\Omega|}}{|\beta|^{\frac{p}{p-1}}}=-(p-1)\frac{P(\Omega)}{|\Omega|^2}\int_\Omega v\,dx,
\]
which is \eqref{neumannlim}. This completes the proof.
\end{proof}

Let us conclude with the following remarks.

\begin{rem}\label{rem:variational_c0}
It is useful to isolate the coefficient appearing in \eqref{neumannlim}. We set
\[
c_0(\Omega):=(p-1)\frac{P(\Omega)}{|\Omega|^2}\int_\Omega v\,dx,
\]
so that Theorem \ref{asymptotic_zero} can be written as
\[
\lambda_p(\beta,\Omega)
=
\beta\frac{P(\Omega)}{|\Omega|}
-
c_0(\Omega)|\beta|^{\frac{p}{p-1}}
+
o\left(|\beta|^{\frac{p}{p-1}}\right),
\qquad \beta\to0.
\]
The coefficient \(c_0(\Omega)\) has a natural variational characterization.
Let
\[
X:=
\left\{
\varphi\in W^{1,p}(\Omega):
\int  _{\de\Omega} \varphi\,d\mathcal H^{N-1}=0
\right\}.
\]
For \(\varphi\in W^{1,p}(\Omega)\), define
\[
\mathcal F_0(\varphi)
:=
  \frac{
\displaystyle\int_\Omega|\nabla\varphi|^p\,dx
+
p\displaystyle\int_{\de\Omega}\varphi\,d\mathcal H^{N-1}
-
p\frac{P(\Omega)}{|\Omega|}\displaystyle\int_\Omega\varphi\,dx
}{|\Omega|}.
\]
Then
\begin{equation}
\label{eq:variational_c0}
-c_0(\Omega)
=
\min_{\varphi\in X}\mathcal F_0(\varphi).
\end{equation}
On \(X\), by the Poincaré inequality with zero
boundary mean, \(\mathcal F_0\) is coercive and strictly convex. Hence it has a
unique minimizer, which satisfies \eqref{equazionelimite}, the Euler--Lagrange equation of \(\mathcal F_0\) on \(X\). Moreover, testing the equation with
\(v\) itself and using the normalization   \(\int_{\de\Omega}v=0\) , we get
\[
\int_\Omega|\nabla v|^p\,dx
=
\frac{P(\Omega)}{|\Omega|}\int_\Omega v\,dx.
\]
Therefore
\[
\mathcal F_0(v)
=
\frac{
|\Omega|\displaystyle\int_\Omega|\nabla v|^p\,dx
-
pP(\Omega)\displaystyle\int_\Omega v\,dx
}{|\Omega|^2}
=
-(p-1)\frac{P(\Omega)}{|\Omega|^2}\int_\Omega v\,dx
=
-c_0(\Omega),
\]
which proves \eqref{eq:variational_c0}.

\end{rem}

\begin{rem}
\label{rem:comparison_BDP}
The previous remark also clarifies the relation between Theorem
\ref{asymptotic_zero} and the estimates obtained in \cite{BDP_pota}. Those
estimates are stated for \(\beta>0\). Under the assumptions required there,
one has
\begin{equation}
\label{eq:BDP_nu_bound}
\frac{1}{\lambda_p(\beta,\Omega)^{\frac{1}{p-1}}}
\le
\frac{1}{\nu_p}
+
\left(\frac{|\Omega|}{\beta P(\Omega)}\right)^{\frac{1}{p-1}},
\end{equation}
where
\[
\nu_p
=
\inf_{\substack{f\in L^\infty(\Omega)\\ f\ge0,\ f\not\equiv0}}
\frac{\displaystyle\int_\Omega f^{p'}\,dx}{\displaystyle\int_\Omega|\nabla v_f|^p\,dx},
\qquad
p'=\frac{p}{p-1},
\]
and \(v_f\) is the normalized solution of
\[
  \begin{cases}
-\Delta_p v_f=f & \text{in }\Omega,\\[1ex]
|\nabla v_f|^{p-2}\dfrac{\de v_f}{\de\nu}
=
-\dfrac1{P(\Omega)}\displaystyle\int_\Omega f\,dx
& \text{on }\de\Omega,\\[2ex]
\displaystyle\int_{\de\Omega}v_f\,d\mathcal H^{N-1}=0.
\end{cases} 
\]
Moreover, if \(u_\Omega\) denotes the \(p\)-torsion function,
\[
  \begin{cases}
-\Delta_p u_\Omega=1 & \text{in }\Omega,\\
u_\Omega=0 & \text{on }\de\Omega
\end{cases} 
\quad\text{ and }\quad T_p(\Omega)
:=
\int_\Omega u_\Omega\,dx
=
\int_\Omega|\nabla u_\Omega|^p\,dx,
\]
then
\begin{equation}
\label{eq:BDP_torsion_bound}
\frac{1}{\lambda_p(\beta,\Omega)^{\frac{1}{p-1}}}
\ge
\frac{T_p(\Omega)}{|\Omega|}
+
\left(\frac{|\Omega|}{\beta P(\Omega)}\right)^{\frac{1}{p-1}}.
\end{equation}

Expanding \eqref{eq:BDP_nu_bound} and \eqref{eq:BDP_torsion_bound} as
\(\beta\to0^+\), one obtains
\[
-\frac{p-1}{\nu_p}
\left(\frac{P(\Omega)}{|\Omega|}\right)^{p'}
\le
\liminf_{\beta\to0^+}
\frac{
\lambda_p(\beta,\Omega)-\beta\frac{P(\Omega)}{|\Omega|}
}{\beta^{p'}}
\]
and
\[
\limsup_{\beta\to0^+}
\frac{
\lambda_p(\beta,\Omega)-\beta\frac{P(\Omega)}{|\Omega|}
}{\beta^{p'}}
\le
-(p-1)\frac{T_p(\Omega)}{|\Omega|}
\left(\frac{P(\Omega)}{|\Omega|}\right)^{p'}.
\]
Consequently,
\begin{equation}
\label{eq:c0_bracket}
(p-1)\frac{T_p(\Omega)}{|\Omega|}
\left(\frac{P(\Omega)}{|\Omega|}\right)^{p'}
\le
c_0(\Omega)
\le
\frac{p-1}{\nu_p}
\left(\frac{P(\Omega)}{|\Omega|}\right)^{p'}.
\end{equation}

The two inequalities in \eqref{eq:c0_bracket} can also be recovered directly
from the variational characterization \eqref{eq:variational_c0}. Indeed, the
torsion bound follows by testing \(\mathcal F_0\) with
\[
z=
\left(\frac{P(\Omega)}{|\Omega|}\right)^{\frac{1}{p-1}}u_\Omega.
\]
Since \(z\in X\), \(u_\Omega=0\) on   \(\de\Omega\) , and
\[
\int_\Omega|\nabla u_\Omega|^p\,dx=\int_\Omega u_\Omega\,dx=T_p(\Omega),
\]
we get
\[
\mathcal F_0(v)
\le
\mathcal F_0(z)
=
-(p-1)\frac{T_p(\Omega)}{|\Omega|}
\left(\frac{P(\Omega)}{|\Omega|}\right)^{p'}.
\]
Since \(\mathcal F_0(v)=-c_0(\Omega)\), this gives the left-hand inequality
in \eqref{eq:c0_bracket}.

On the other hand, choosing the constant datum
\[
f=\frac{P(\Omega)}{|\Omega|}
\]
in the definition of \(\nu_p\), the corresponding function \(v_f\) is 
the solution \(v\) of \eqref{equazionelimite}. Hence
\[
\frac{1}{\nu_p}
\ge
\frac{\displaystyle\int_\Omega|\nabla v|^p\,dx}{\displaystyle\int_\Omega f^{p'}\,dx}.
\]
Using
\[
\int_\Omega|\nabla v|^p\,dx
=
\frac{P(\Omega)}{|\Omega|}\int_\Omega v\,dx
\]
and
\[
\int_\Omega f^{p'}\,dx
=
|\Omega|
\left(\frac{P(\Omega)}{|\Omega|}\right)^{p'},
\]
we obtain
\[
\frac{1}{\nu_p}
\ge
\left(\frac{P(\Omega)}{|\Omega|}\right)^{-\frac{1}{p-1}}
\frac{\displaystyle\int_\Omega v\,dx}{|\Omega|}.
\]
Equivalently,
\[
c_0(\Omega)
=
(p-1)\frac{P(\Omega)}{|\Omega|^2}\int_\Omega v\,dx
\le
\frac{p-1}{\nu_p}
\left(\frac{P(\Omega)}{|\Omega|}\right)^{p'},
\]
which is the right-hand inequality in \eqref{eq:c0_bracket}.
\end{rem}
\appendix
\section{Appendix}\label{appendix:properties}
\subsection{Properties of the first Robin eigenvalue}
\begin{proof}[Proof of Lemma \ref{traceineq}]
We argue by contradiction. If \eqref{trace:interp} fails for some $\eps_0>0$, there exist $w_m\in W^{1,p}(\Omega)$ such that
\[
\eps_0\int_\Omega\abs{\nabla w_m}^p\,dx+m\int_\Omega\abs{w_m}^p\,dx<\int_{\de\Omega}\abs{w_m}^p\,d\mathcal H^{N-1}=1.
\]
Then $\set{w_m}$ is bounded in $W^{1,p}(\Omega)$ and, up to a subsequence, $w_m\rightharpoonup w$ in $W^{1,p}(\Omega)$, $w_m\to w$ in $L^p(\Omega)$ and, by the compactness of the trace operator, also in $L^p(\de\Omega)$. From $\int_\Omega\abs{w_m}^p\,dx<1/m$ we get $w=0$ a.e. in $\Omega$, whence its trace vanishes; this contradicts $\int_{\de\Omega}\abs{w}^p\,d\mathcal H^{N-1}=\lim_m\int_{\de\Omega}\abs{w_m}^p\,d\mathcal H^{N-1}=1$.
\end{proof}

\begin{proof}[Proof of Proposition \ref{prop:properties}]
\emph{Existence and sign.} Testing \eqref{variationalaut} with the first Dirichlet eigenfunction  gives $\lambda_p(\beta,\Omega)\le\lambda_p^D(\Omega)$.

If $\beta\ge0$ the Rayleigh quotient is nonnegative, hence $\lambda_p(\beta,\Omega)\ge0$. If $\beta<0$, by Lemma \ref{traceineq} with $\eps=\frac{1}{2\abs{\beta}}$, for every $w\in W^{1,p}(\Omega)$,
\begin{multline*}
\int_\Omega\abs{\nabla w}^p\,dx+\beta\int_{\de\Omega}\abs{w}^p\,d\mathcal H^{N-1}\ge \\ \ge \tfrac12\int_\Omega\abs{\nabla w}^p\,dx-\abs{\beta}C(\eps)\int_\Omega\abs{w}^p\,dx\ge -\abs{\beta}C(\eps)\int_\Omega\abs{w}^p\,dx,
\end{multline*}
so that $\lambda_p(\beta,\Omega)\ge -\abs{\beta}C(\eps)>-\infty$.

Let $\set{w_n}\subseteq W^{1,p}(\Omega)$ be a minimizing sequence with $\norm{w_n}_{L^p(\Omega)}=1$. The above estimates show that $\int_\Omega\abs{\nabla w_n}^p\,dx$ is bounded, hence $\set{w_n}$ is bounded in $W^{1,p}(\Omega)$. Up to a subsequence, $w_n\rightharpoonup u_\beta$ in $W^{1,p}(\Omega)$, $w_n\to u_\beta$ strongly in $L^p(\Omega)$ and, by compactness of the trace embedding, in $L^p(\de\Omega)$. Therefore $\norm{u_\beta}_{L^p(\Omega)}=1$, $\int_{\de\Omega}\abs{u_\beta}^p=\lim_n\int_{\de\Omega}\abs{w_n}^p$ and, by the weak lower semicontinuity of $w\mapsto\int_\Omega\abs{\nabla w}^p$,
\begin{multline*}
\int_\Omega\abs{\nabla u_\beta}^p\,dx+\beta\int_{\de\Omega}\abs{u_\beta}^p\,d\mathcal H^{N-1}\le \\ \le\liminf_n\left(\int_\Omega\abs{\nabla w_n}^p\,dx+\beta\int_{\de\Omega}\abs{w_n}^p\,d\mathcal H^{N-1}\right)=\lambda_p(\beta,\Omega),
\end{multline*}
so that $u_\beta$ is a minimizer.

Finally, the constant function $w\equiv1$ gives $\lambda_p(\beta,\Omega)\le \beta\frac{P(\Omega)}{\abs{\Omega}}$, which is negative for $\beta<0$, while $\lambda_p(0,\Omega)=0$. For $\beta>0$, if it were $\lambda_p(\beta,\Omega)=0$, a normalized minimizer would satisfy $\int_\Omega\abs{\nabla u_\beta}^p\,dx=0$ and $\int_{\de\Omega}\abs{u_\beta}^p\,d\mathcal H^{N-1}=0$, i.e. $u_\beta$ would be a constant of zero trace, hence $u_\beta\equiv0$, against $\norm{u_\beta}_{L^p(\Omega)}=1$. Thus $\lambda_p(\beta,\Omega)>0$ for $\beta>0$.

\emph{Simplicity.} Let $u$ be a minimizer. Since $\int_\Omega\abs{\nabla\abs{u}}^p\,dx=\int_\Omega\abs{\nabla u}^p\,dx$, $|u|$ is a minimizer as well. By the Harnack inequality {(see \cite{trudinger})} and the connectedness of $\Omega$, $\abs{u}>0$ in $\Omega$; thus every minimizer has constant sign.

Let now $u_1,u_2>0$ be two minimizers normalized by $\norm{u_1}_{L^p(\Omega)}=\norm{u_2}_{L^p(\Omega)}=1$, and for $t\in[0,1]$ set
\[
\sigma_t=\big((1-t)u_1^p+t\,u_2^p\big)^{1/p}\in W^{1,p}(\Omega).
\]
It holds that {(see \cite{lindqvist})}
\begin{equation}\label{hidden}
\int_\Omega\abs{\nabla\sigma_t}^p\,dx\le(1-t)\int_\Omega\abs{\nabla u_1}^p\,dx+t\int_\Omega\abs{\nabla u_2}^p\,dx,
\end{equation}
with equality if and only if $u_1$ and $u_2$ are proportional. Since $\sigma_t^p=(1-t)u_1^p+t\,u_2^p$ pointwise, $\int_\Omega\sigma_t^p\,dx=1$ and the boundary term satisfies
\[
\int_{\de\Omega}\abs{\sigma_t}^p\,d\mathcal H^{N-1}=(1-t)\int_{\de\Omega}\abs{u_1}^p\,d\mathcal H^{N-1}+t\int_{\de\Omega}\abs{u_2}^p\,d\mathcal H^{N-1}.
\]
Therefore, for any $\beta\in\R$,
\[
\lambda_p(\beta,\Omega)\le \int_\Omega\abs{\nabla\sigma_t}^p\,dx+\beta\int_{\de\Omega}\abs{\sigma_t}^p\,d\mathcal H^{N-1}\le(1-t)\lambda_p(\beta,\Omega)+t\,\lambda_p(\beta,\Omega)=\lambda_p(\beta,\Omega).
\]
All inequalities are thus equalities; in particular equality holds in \eqref{hidden}, which forces $u_1$ and $u_2$ to be proportional. By the normalization, $u_1=u_2$.

\emph{Concavity and limits.} For fixed $w$, the map $\beta\mapsto\frac{\int_\Omega\abs{\nabla w}^p\,dx+\beta\int_{\de\Omega}\abs{w}^p\,d\mathcal H^{N-1}}{\int_\Omega\abs{w}^p\,dx}$ is affine with nonnegative slope; hence $\lambda_p(\cdot,\Omega)$, being the pointwise infimum of such a family, is concave, upper semicontinuous and non-decreasing. Being concave and finite on $\R$, it is continuous on $\R$; together with $\lambda_p(0,\Omega)=0$, this gives $\lim_{\beta\to0}\lambda_p(\beta,\Omega)=0$.

\emph{Limits as $\beta\to\pm\infty$ and strong convergence in the
Dirichlet limit.}
The constant test function gives
 \[
\lambda_p(\beta,\Omega)
\leq
\beta\frac{P(\Omega)}{|\Omega|}
\longrightarrow-\infty
\qquad\text{as }\beta\to-\infty.
\]

Let now $\beta\to+\infty$. Since the map
$\beta\mapsto\lambda_p(\beta,\Omega)$ is non-decreasing and bounded
above by $\lambda_p^D(\Omega)$,  there exists
$L\leq\lambda_p^D(\Omega)$ such that
\[
\lambda_p(\beta,\Omega)\longrightarrow L.
\]
Testing the eigenvalue equation with $u_\beta$ and recalling that
$\|u_\beta\|_{L^p(\Omega)}=1$, we obtain
\begin{equation}
\label{eq:energy_identity_dirichlet_limit}
\int_\Omega|\nabla u_\beta|^p\,dx
+
\beta\int_{\de\Omega}u_\beta^p\,d\mathcal H^{N-1}
=
\lambda_p(\beta,\Omega)
\leq
\lambda_p^D(\Omega).
\end{equation}
Consequently,
$\int_{\de\Omega}u_\beta^p\,d\mathcal H^{N-1}
\leq
\frac{\lambda_p^D(\Omega)}{\beta}
\longrightarrow0$ and the family $\{u_\beta\}$ is bounded in $W^{1,p}(\Omega)$.

Let $\beta_n\to+\infty$. Up to a subsequence, not relabeled, there
exists $u_*\in W^{1,p}(\Omega)$ such that
\[
u_{\beta_n}\rightharpoonup u_*
\text{ weakly in }W^{1,p}(\Omega),
\qquad
u_{\beta_n}\longrightarrow u_*
\text{ strongly in }L^p(\Omega) \text{ and in }L^p(\de\Omega).
\]
 It follows that $\|u_*\|_{L^p(\Omega)}=1$
and $u_*=0$ on $\de\Omega$, so that $u_*\in W^{1,p}_0(\Omega)$. Hence, by the variational
characterization of the first Dirichlet eigenvalue and the weak lower
semicontinuity of the Dirichlet integral,
\begin{align*}
\lambda_p^D(\Omega)
\leq
\int_\Omega|\nabla u_*|^p\,dx
\leq
\liminf_{n\to\infty}
\int_\Omega|\nabla u_{\beta_n}|^p\,dx
&\leq
\limsup_{n\to\infty}
\int_\Omega|\nabla u_{\beta_n}|^p\,dx
\\
&\leq
\lim_{n\to\infty}\lambda_p(\beta_n,\Omega)
=
L\le\lambda_p^D(\Omega).
\end{align*}
Hence
$L=\lambda_p^D(\Omega)$ , and $u_*$ is a normalized first Dirichlet eigenfunction. Since
$u_{\beta_n}\geq0$ and $u_{\beta_n}\to u_*$ strongly in
$L^p(\Omega)$, we have $u_*\geq0$. By the simplicity of the first
Dirichlet eigenvalue,
$u_*=u_\infty$.
Moreover,
\[
\int_\Omega|\nabla u_{\beta_n}|^p\,dx
\longrightarrow
\int_\Omega|\nabla u_\infty|^p\,dx.
\]
Together with the weak convergence in $W^{1,p}$ yields
\[
u_{\beta_n}\longrightarrow u_\infty
\quad\text{strongly in }W^{1,p}(\Omega).
\]

As every sequence $\beta_n\to+\infty$ admits a subsequence converging
strongly to the same limit $u_\infty$, the whole family satisfies
\begin{equation*}
u_\beta\longrightarrow u_\infty
\quad\text{strongly in }W^{1,p}(\Omega)
\qquad\text{as }\beta\to+\infty.
\end{equation*}

\emph{Continuity of $\beta\mapsto u_\beta$.} We reason similarly as in the previous step. Let $\beta_n\to\beta_0$. The minimizers $u_{\beta_n}$ are bounded in $W^{1,p}(\Omega)$ (uniformly for $\beta_n$ in a bounded interval), so $u_{\beta_n}\rightharpoonup \bar u$ in $W^{1,p}(\Omega)$, strongly in $L^p(\Omega)$ and $L^p(\de\Omega)$. By lower semicontinuity and the continuity of $\lambda_p(\cdot,\Omega)$,
\begin{multline*}
\int_\Omega\abs{\nabla\bar u}^p\,dx+\beta_0\int_{\de\Omega}\bar u^p\,d\mathcal H^{N-1}\le\liminf_n\Big(\int_\Omega\abs{\nabla u_{\beta_n}}^p\,dx+\beta_n\int_{\de\Omega}u_{\beta_n}^p\,d\mathcal H^{N-1}\Big)=\\=\lim_n\lambda_p(\beta_n,\Omega)=\lambda_p(\beta_0,\Omega),
\end{multline*}
so $\bar u$ is a minimizer for $\beta_0$, and by the uniqueness of the minimizer $\bar u=u_{\beta_0}$. Moreover the above chain is an equality, whence $\int_\Omega\abs{\nabla u_{\beta_n}}^p\,dx\to\int_\Omega\abs{\nabla u_{\beta_0}}^p\,dx$; together with the weak convergence in  $L^p$, this gives  $u_{\beta_n}\to u_{\beta_0}$ in $W^{1,p}(\Omega)$. Since the limit does not depend on the subsequence, the whole family converges.

\emph{Differentiability.} Using $u_\beta$ as a test function for $\lambda_p(\beta+h,\Omega)$ and recalling $\norm{u_\beta}_{L^p(\Omega)}=1$,
\[
\lambda_p(\beta+h,\Omega)\le\int_\Omega\abs{\nabla u_\beta}^p\,dx+(\beta+h)\int_{\de\Omega}u_\beta^p\,d\mathcal H^{N-1}=\lambda_p(\beta,\Omega)+h\int_{\de\Omega}u_\beta^p\,d\mathcal H^{N-1},
\]
so $\limsup_{h\to0^+}\frac{\lambda_p(\beta+h,\Omega)-\lambda_p(\beta,\Omega)}{h}\le\int_{\de\Omega}u_\beta^p\,d\mathcal H^{N-1}$. Symmetrically, testing $\lambda_p(\beta,\Omega)$ with $u_{\beta+h}$ gives
\[
\lambda_p(\beta,\Omega)\le\lambda_p(\beta+h,\Omega)-h\int_{\de\Omega}u_{\beta+h}^p\,d\mathcal H^{N-1},
\]
hence $\frac{\lambda_p(\beta+h,\Omega)-\lambda_p(\beta,\Omega)}{h}\ge\int_{\de\Omega}u_{\beta+h}^p\,d\mathcal H^{N-1}\to\int_{\de\Omega}u_\beta^p\,d\mathcal H^{N-1}$ as $h\to0^+$, by the continuity of $\beta\mapsto u_\beta$ and the compactness of the trace. The same holds for $h\to0^-$; therefore $\lambda_p(\cdot,\Omega)$ is differentiable with derivative \eqref{derivata}, which is in turn continuous, so that $\lambda_p(\cdot,\Omega)\in C^1(\R)$.

\emph{Strict monotonicity.} We first show that $\int_{\de\Omega}u_\beta^p\,d\mathcal H^{N-1}>0$ for every $\beta\in\R$. For $\beta=0$ the eigenfunction is constant and the claim is immediate, so assume $\beta\neq0$. Testing the weak formulation of \eqref{eigenvalue_problem} with the constant function $\varphi\equiv1\in W^{1,p}(\Omega)$, the gradient term vanishes and we obtain
\[
\beta\int_{\de\Omega}u_\beta^{p-1}\,d\mathcal H^{N-1}=\lambda_p(\beta,\Omega)\int_\Omega u_\beta^{p-1}\,dx.
\]
Since $u_\beta>0$ in $\Omega$ and $\lambda_p(\beta,\Omega)$ has the same sign as $\beta$, the right-hand side is nonzero; hence $\int_{\de\Omega}u_\beta^{p-1}\,d\mathcal H^{N-1}\neq0$, so the trace of $u_\beta$ does not vanish $\mathcal H^{N-1}$-a.e. on $\de\Omega$ and therefore $\int_{\de\Omega}u_\beta^p\,d\mathcal H^{N-1}>0$. Recalling \eqref{derivata}, we obtain that the map $\beta\mapsto\lambda_p(\beta,\Omega)$ is strictly increasing. {In particular, since $\lambda_p(\beta,\Omega)\to\lambda_p^D(\Omega)$ as $\beta\to+\infty$, the strict monotonicity yields $\lambda_p(\beta,\Omega)<\lambda_p^D(\Omega)$ for every $\beta\in\R$, which proves the upper inequality in \eqref{bound:sign}.} 
\end{proof}

\subsection{An elementary algebraic inequality}\label{appendix:algebraic}
{
We collect here an elementary algebraic inequality that is repeatedly used throughout the paper to control the remainder terms in the expansion of $p$-Laplacian type operators. For the sake of completeness, we provide a short  proof.

\begin{lem}\label{lem:algebraic_taylor}
Let $p > 1$. There exists a constant $C = C(p) > 0$ such that for all $\xi, \eta \in \mathbb{R}^N$, there holds
\begin{equation}\label{eq:alg_p_geq_2}
\big| |\xi+\eta|^p - |\xi|^p - p |\xi|^{p-2}\xi \cdot \eta \big| \le C \big( |\xi|^{p-2}|\eta|^2 + |\eta|^p \big) \qquad \text{if } p \ge 2,
\end{equation}
and
\begin{equation}\label{eq:alg_p_less_2}
\big| |\xi+\eta|^p - |\xi|^p - p |\xi|^{p-2}\xi \cdot \eta \big| \le C |\eta|^p \qquad \text{if } 1 < p < 2.
\end{equation}
\end{lem}

\begin{proof}
Let us assume $\xi \neq 0$. Factoring out $|\xi|$ and writing $\eta = |\xi| w$ for some $w \in \mathbb{R}^N$, the left-hand side of both inequalities takes the form $|\xi|^p g(v, w)$, where $v = \frac{\xi}{|\xi|} \in S^{N-1}$ and $g(v, w) := \big| |v+w|^p - 1 - p v\cdot w \big|$.

By a Taylor expansion of the map $x \mapsto |x|^p$ centered at $v$, we have $g(v,w) = O(|w|^2)$ uniformly as $|w| \to 0$. Moreover, $g(v,w) = O(|w|^p)$ uniformly as $|w| \to +\infty$.

Assume first $1<p<2$. The ratio $\frac{g(v,w)}{|w|^p}$ for $w \neq 0$ satisfies
\[ \lim_{|w| \to 0} \frac{g(v,w)}{|w|^p} = 0, \qquad \lim_{|w| \to +\infty} \frac{g(v,w)}{|w|^p} = 1, \]
uniformly with respect to $v \in S^{N-1}$. 
Therefore $g(v,w) \le C|w|^p$ for some $C=C(p)>0$, which multiplied by $|\xi|^p$ immediately yields \eqref{eq:alg_p_less_2}.

Assume now $p \ge 2$. In this case, we consider the ratio $\frac{g(v,w)}{|w|^2+|w|^p}$. Computing the limits, we find that
\[ \limsup_{|w| \to 0} \frac{g(v,w)}{|w|^2+|w|^p} \le C', \qquad \lim_{|w| \to +\infty} \frac{g(v,w)}{|w|^2+|w|^p} = 1. \]
By the same uniform continuity arguments, there exists a uniform constant $C>0$ such that $g(v,w) \le C(|w|^2+|w|^p)$. Multiplying by $|\xi|^p$ yields  \eqref{eq:alg_p_geq_2}.
\end{proof}
}
 \section*{Acknowledgement}
This work has been partially supported by GNAMPA of INdAM.
\addcontentsline{toc}{section}{Bibliography}
\bibliographystyle{plain}
\bibliography{bibliografia}
\end{document}